\documentclass[a4paper,12pt,reqno]{amsart}
\usepackage{latexsym}
\usepackage{amssymb,amsfonts,amsmath,mathrsfs,graphicx,multirow}
\usepackage{caption}
\newtheorem{theorem}{Theorem}
\newtheorem{proposition}{Proposition}
\newtheorem{lemma}{Lemma}
\newtheorem{conjecture}{Conjecture}
\newtheorem{remark}{Remark}
\newtheorem{corollary}{Corollary}
\theoremstyle{definition}

\begin{document}

\title{On the derivatives of Hardy's function $Z(t)$}
\author{Hung M. Bui and R. R. Hall}
\subjclass[2010]{11M06, 11M26.}
\keywords{Riemann zeta-function, Hardy's $Z$-function, moments.}
\address{Department of Mathematics, University of Manchester, Manchester M13 9PL, UK}
\email{hung.bui@manchester.ac.uk}
\address{Department of Mathematics, University of York, York YO10 5DD, UK}
\email{richardroxbyhall@gmail.com}

\begin{abstract}
Let $Z^{(k)}(t)$ be the $k$-th derivative of Hardy's $Z$-function. The numerics seem to suggest that if $k$ and $\ell$ have the same parity, then the zeros of $Z^{(k)}(t)$ and $Z^{(\ell)}(t)$ come in pairs which are very close to each other. That is to say that $Z^{(k)}(t)Z^{(\ell)}(t)$ has constant sign for the majority, if not almost all, of values $t$. In this paper we show that this is true a positive proportion of times. We also study the sign of the product of four derivatives of Hardy's function, $Z^{(k)}(t)Z^{(\ell)}(t)Z^{(m)}(t)Z^{(n)}(t)$.
\end{abstract}

\allowdisplaybreaks

\maketitle

\section{Introduction}

We are interested in the sign of each of the functions
\begin{align*}
& (\text{A})\qquad Z^{(k)}(t)Z^{(\ell)}(t),\\
& (\text{B})\qquad Z^{(k)}(t)Z^{(\ell)}(t)Z^{(m)}(t)Z^{(n)}(t),
\end{align*}
in which $Z^{(k)}(t)$ is the $k$-th derivative of Hardy's function
\begin{align*}
Z(t)&:=e^{i\theta(t)}\zeta(\tfrac12+it)\\
&=\bigg(\pi^{-it}\frac{\Gamma(\frac14+\frac{it}{2})}{\Gamma(\frac14-\frac{it}{2})}\bigg)^{1/2}\zeta(\tfrac12+it).
\end{align*} 
This is related to the work of Gonek and Ivi\'c [\textbf{\ref{GI}}], where the sign of $Z(t)$ is studied.

For case (A), if $k+\ell =2s+1$ is odd, then we have that
\[
\int_0^TZ^{(k)}(t)Z^{(\ell)}(t)dt\ll_{k,\ell}T^{1/3},
\]
by employing the bound $Z^{(k)}(t)\ll_kT^{1/6}$ and integration by parts. When $k+\ell=2s$ is even and $|k-\ell|=2d$ we have the mean values
\begin{equation}\label{meanZkZl}
\int_{0}^TZ^{(k)}(t)Z^{(\ell)}(t)dt=\frac{(-1)^d}{4^s(2s+1)}TQ_{2s+1}\Big(\log\frac{T}{2\pi}\Big)+O\big(T^{3/4}(\log T)^{2s+1/2}\big),
\end{equation}
where $Q_{2s+1}(x)$ is a monic polynomial of degree $2s+1$. This follows from [\textbf{\ref{H}}; Theorem 3] by integration by parts. In the simplest case $k=0$ and $\ell=2$, formula \eqref{meanZkZl} shows that there exist many values of $t$ such that $Z(t)Z''(t)<0$. We draw the reader's attention to Figures 1--2 below which suggest that in fact the majority, if not almost all, of values $t$ have this property. 

\vspace{0.3cm}

\begin{center}
\includegraphics[width=0.85\textwidth]{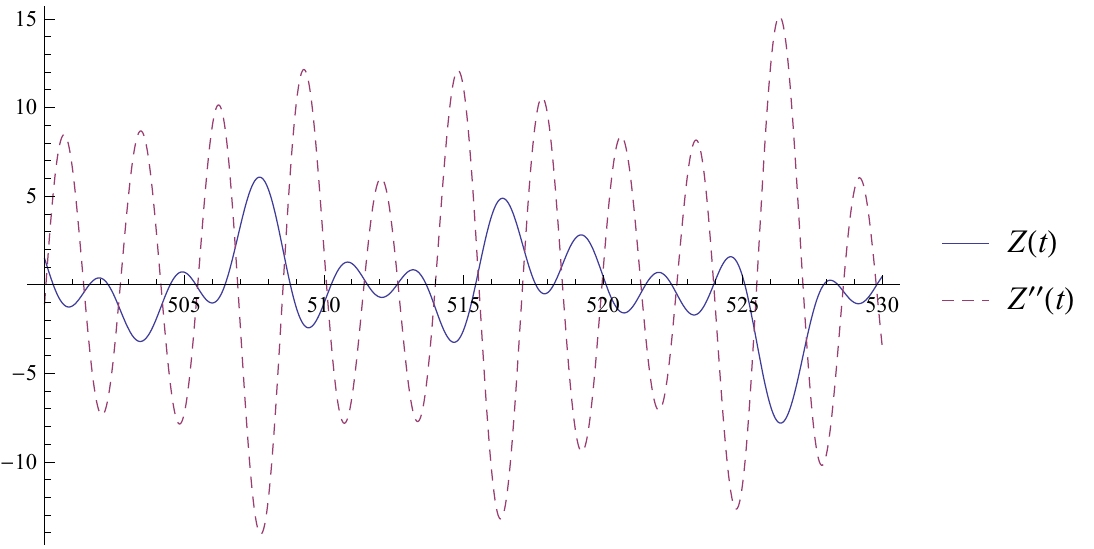}

\vspace{0.2cm}
\textsc{Figure 1}. $Z(t)$ and $Z''(t)$ for $t\in[500,530]$.
\end{center}
\vspace{0.3cm}

\begin{center}
\includegraphics[width=0.85\textwidth]{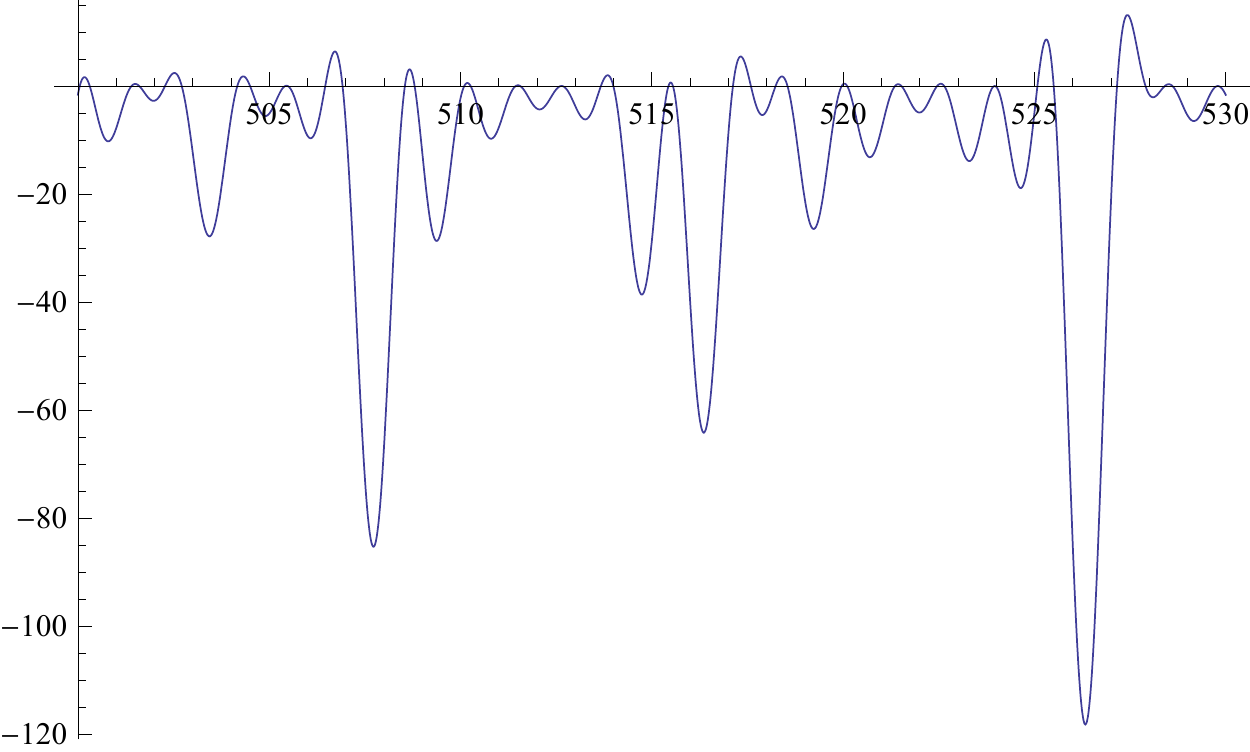}

\vspace{0.2cm}
\textsc{Figure 2}. $Z(t)Z''(t)$ for $t\in[500,530]$.
\end{center}
\vspace{0.3cm}

Out first theorem shows that this is true more than $12\%$ of the times.

\begin{theorem}\label{thm1}
We have
\[
\emph{meas}\big\{t\in[T,2T]:Z(t)Z''(t)< 0\big\}\geq \Big(\frac{3}{25}+o(1)\Big)T.
\]
\end{theorem}

In general, we deduce from \eqref{meanZkZl} that the integrand has sign $(-1)^d$ for many values of $t$. In fact, in view of Figures 3--4 below, this seems to be the case for almost all $t$ as we increase $k$ and $\ell$. That is to say if $k$ and $\ell$ have the same parity, then the zeros of $Z^{(k)}(t)$ and $Z^{(\ell)}(t)$ come in pairs which are very close to each other. 

\vspace{0.3cm}

\begin{center}
\includegraphics{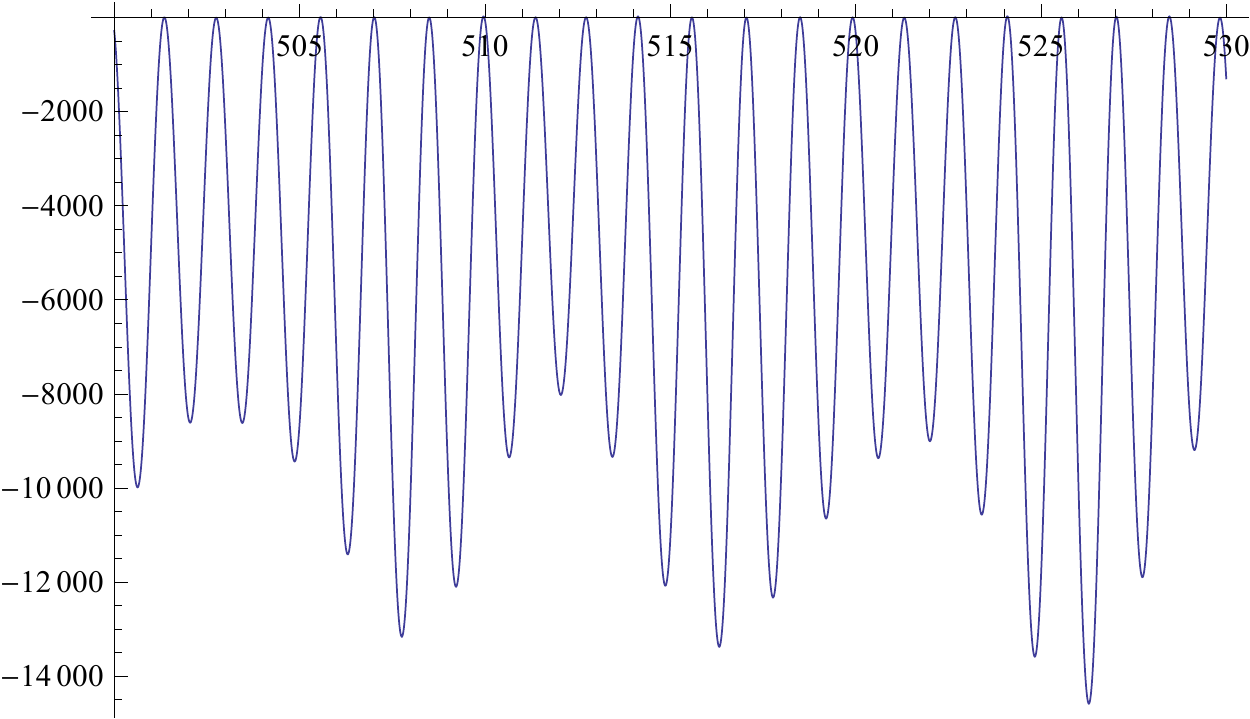}

\vspace{0.2cm}
\textsc{Figure 3}. $Z^{(4)}(t)Z^{(6)}(t)$ for $t\in[500,530]$.
\end{center}
\vspace{0.3cm}

\begin{center}
\includegraphics{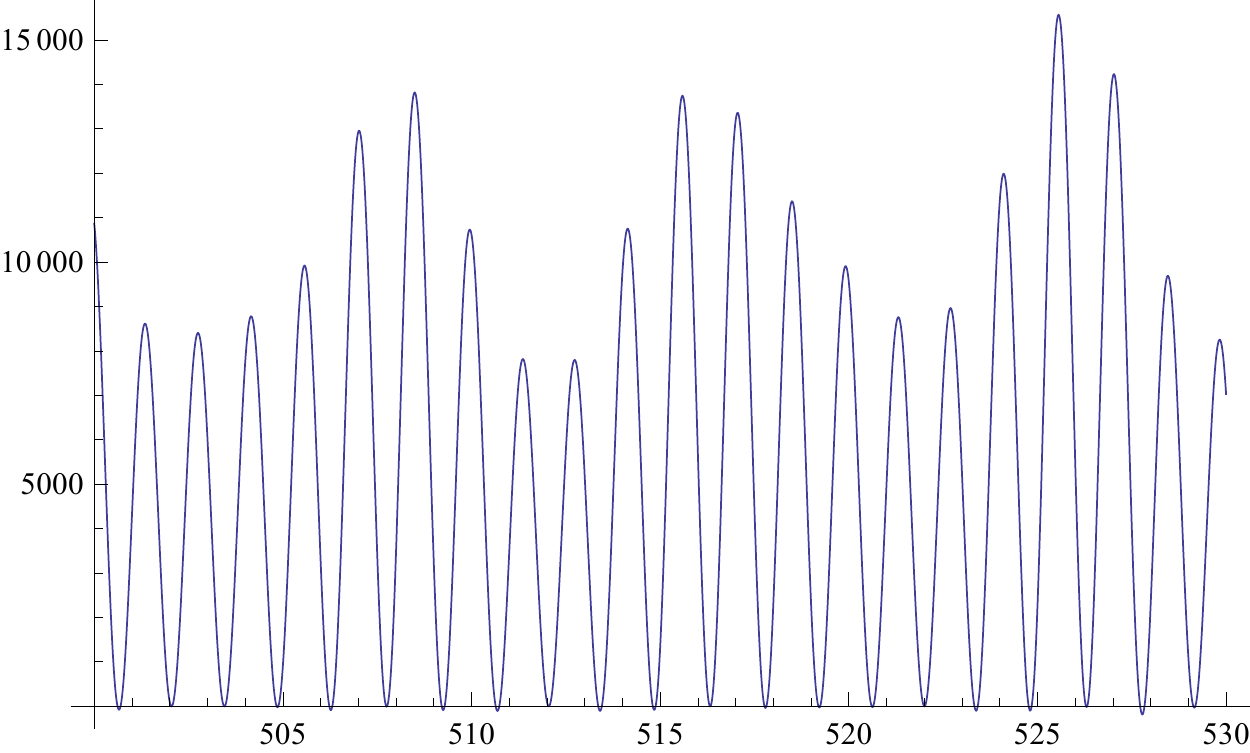}

\vspace{0.2cm}
\textsc{Figure 4}. $Z'''(t)Z^{(7)}(t)$ for $t\in[500,530]$.
\end{center}
\vspace{0.3cm}

Our next theorem shows that $Z^{(k)}(t)Z^{(\ell)}(t)$ has sign $(-1)^d$ a positive proportion of times.

\begin{theorem}\label{thm2}
Let $k+\ell$ be even and
\begin{align*}
S_{k,\ell}^+&:=\big\{t\in[T,2T]:Z^{(k)}(t)Z^{(\ell)}(t)> 0\big\},\\
S_{k,\ell}^-&:=\big\{t\in[T,2T]:Z^{(k)}(t)Z^{(\ell)}(t)< 0\big\}.
\end{align*}
Then if $k\equiv \ell\,(\emph{mod}\ 4)$, we have
\[
\emph{meas}(S_{k,\ell}^+)\gg_{k,\ell} T,
\]
and if $k\equiv \ell+2\,(\emph{mod}\ 4)$, we have
\[
\emph{meas}(S_{k,\ell}^-)\gg_{k,\ell} T.
\]
\end{theorem}

\begin{remark}
\emph{We notice that these congruence conditions may be written in the equivalent forms
\begin{equation}\label{condition1}
|i^k+i^\ell|=2\qquad\text{and}\qquad |i^k+i^\ell|=0,
\end{equation}
respectively. When we consider case (B), a condition similar to \eqref{condition1} will be seen to be relevant.}
\end{remark}

\begin{remark}
\emph{It is known that the zeros of the $k$-th derivative of the  Riemann $\Xi$-function, where
\[
\Xi(s)=\frac{s(s-1)}{2}\pi^{-s/2}\Gamma\Big(\frac s2\Big)\zeta(s),
\]
become evenly spaced out as $k\rightarrow\infty$. This was conjectured by Farmer and Rhoades [\textbf{\ref{FR}}], and later proved by Ki [\textbf{\ref{K}}] (see also [\textbf{\ref{GH}}]). These works provide a rough heuristic that if $k,\ell\rightarrow\infty$ and have the same parity, then the zeros of $\Xi^{(k)}(t)$ and $\Xi^{(\ell)}(t)$ become pairwise close. The difference in our Theorem \ref{thm1} and Theorem \ref{thm2} to the mentioned works is that the order of differentiation is fixed. For $Z(t)$, Figures 1--4 above seem to suggest that the zeros of $Z^{(k)}(t)$ and $Z^{(\ell)}(t)$ are pairwise very close even for small $k,\ell$ with the same parity.
}
\end{remark}

For case (B), we define the function $\text{HARDY}(k,\ell,m,n)$ via the formula
\begin{align*}
\int_0^T Z^{(k)}(t)Z^{(\ell)}(t)Z^{(m)}(t)Z^{(n)}(t)dt&=\frac{\text{HARDY}(k,\ell,m,n)}{\pi^2}T(\log T)^{k+\ell+m+n+4}\\
&\qquad+O\big(T(\log T)^{k+\ell+m+n+3}\big).
\end{align*}
This is a rational function and is given explicitly by the following result.

\begin{theorem}\label{thmHardy}
We have
\begin{align*}
\emph{HARDY}(k,\ell,m,n)&=(-1)^{m+n}i^{k+\ell+m+n}3\int_0^1\int_0^1\int_0^1\int_0^1(u_1-u_2)^2\\
&\ \times \Big(\frac12+(u_1-u_2)u_3-u_1\Big)^k\Big(\frac12+(u_2-u_1)u_3-u_2\Big)^\ell\\
&\ \times \Big(\frac12+(u_1-u_2)u_4-u_1\Big)^m\Big(\frac12+(u_2-u_1)u_4-u_2\Big)^ndu_1du_2du_3du_4.
\end{align*}
\end{theorem}

A consequence of Theorem \ref{thmHardy} is

\begin{corollary}\label{corHardy} 
If $k+\ell+m+n$ is odd, then we have $\emph{HARDY}(k,\ell,m,n)=0$.
\end{corollary}

In the case $k+\ell+m+n$ is even, the sign of the function $\text{HARDY}(k,\ell,m,n)$ is not well-understood. The tables below show its values when $k+\ell+m+n=6,8$ or $10$.

\vspace{3mm}
\begin{table}[h!]
\centering
\begin{tabular}{ |c|c||c|c||c|c|  }
 \hline
 $(k,\ell,m,n)$ &  & $(k,\ell,m,n)$ &  & $(k,\ell,m,n)$ & \\
 \hline\hline
$(6,0,0,0)$   & $-\frac{1}{2688}$  & $(4,1,1,0)$ &   $\frac{1}{40320}$ & $(3,1,1,1)$ & $-\frac{19}{201600}$\\[3pt]
$(5,1,0,0)$  &   $\frac{1}{8064}$ & $(3,3,0,0)$  & $\frac{1}{9600}$ & $(2,2,2,0)$ & $-\frac{61}{604800}$\\[3pt]
$(4,2,0,0)$ & $-\frac{1}{5760}$ & $(3,2,1,0)$ &  $\frac{1}{28800}$ & $(2,2,1,1)$ & $\frac{19}{604800}$\\[3pt]
 \hline
\end{tabular}
\caption{HARDY$(k,\ell,m,n)$ with $k+\ell+m+n=6$ and $k\geq\ell\geq m\geq n$.}
\end{table}

\begin{table}[h!]
\centering
\begin{tabular}{ |c|c||c|c||c|c|  }
 \hline
 $(k,\ell,m,n)$ &  & $(k,\ell,m,n)$ &  & $(k,\ell,m,n)$ & \\
 \hline\hline
$(8,0,0,0)$   & $\frac{1}{16896}$  & $(5,2,1,0)$ &   $-\frac{1}{197120}$ & $(4,2,1,1)$ & $-\frac{23}{5322240}$\\[3pt]
$(7,1,0,0)$  &   $-\frac{1}{50688}$ & $(5,1,1,1)$  & $\frac{23}{1774080}$ & $(3,3,2,0)$ & $-\frac{41}{8870400}$\\[3pt]
$(6,2,0,0)$ & $\frac{1}{39424}$ & $(4,4,0,0)$ &  $\frac{3}{140800}$ & $(3,3,1,1)$ & $\frac{19}{1774080}$\\[3pt]
$(6,1,1,0)$ & $-\frac{1}{354816}$ & $(4,3,1,0)$ &  $-\frac{3}{895600}$ & $(3,2,2,1)$ & $-\frac{17}{5322240}$\\[3pt]
$(5,3,0,0)$ & $-\frac{3}{197120}$ & $(4,2,2,0)$ &  $\frac{331}{26611240}$ & $(2,2,2,2)$ & $\frac{17}{1774080}$\\[3pt]
 \hline
\end{tabular}
\caption{HARDY$(k,\ell,m,n)$ with $k+\ell+m+n=8$ and $k\geq\ell\geq m\geq n$.}
\end{table}

\newpage

\begin{table}[h!]
\centering
\begin{tabular}{ |c|c||c|c||c|c|  }
 \hline
 $(k,\ell,m,n)$ &  & $(k,\ell,m,n)$ &  & $(k,\ell,m,n)$ & \\
 \hline\hline
$(10,0,0,0)$   & $-\frac{3}{292864}$  & $(6,3,1,0)$ &   $\frac{1}{2795520}$ & $(4,4,2,0)$ & $-\frac{571}{358758400}$\\[3pt]
$(9,1,0,0)$  &   $\frac{1}{292864}$ & $(6,2,2,0)$  & $-\frac{-173}{92252160}$ & $(4,4,1,1)$ & $\frac{241}{358758400}$\\[3pt]
$(8,2,0,0)$ & $-\frac{1}{239616}$ & $(6,2,1,1)$ &  $\frac{1}{1464320}$ & $(4,3,3,0)$ & $\frac{467}{1076275200}$\\[3pt]
$(8,1,1,0)$ & $\frac{1}{2635776}$ & $(5,5,0,0)$ &  $\frac{3}{1304576}$ & $(4,3,2,1)$ & $\frac{1403}{3228825600}$\\[3pt]
$(7,3,0,0)$ & $\frac{1}{399360}$ & $(5,4,1,0)$ &  $\frac{3}{6522880}$ & $(4,2,2,2)$ & $-\frac{127}{97843200}$\\[3pt]
$(7,2,1,0)$ & $\frac{1}{1198080}$ & $(5,3,2,0)$ &  $\frac{467}{645765120}$ & $(3,3,3,1)$ & $-\frac{467}{358758400}$\\[3pt]
$(7,1,1,1)$ & $-\frac{3}{1464320}$ & $(5,3,1,1)$ &  $-\frac{-199}{129153024}$ & $(3,3,2,2)$ & $\frac{127}{293529600}$\\[3pt]
$(6,4,0,0)$ & $-\frac{-3}{931840}$ & $(5,2,2,1)$ &  $\frac{277}{645765120}$ &  & \\[3pt]
 \hline
\end{tabular}
\caption{HARDY$(k,\ell,m,n)$ with $k+\ell+m+n=10$ and $k\geq\ell\geq m\geq n$.}
\end{table}

The function $\text{HARDY}(k,\ell,m,n)$ was also evaluated by the second author in [\textbf{\ref{H1}}] for some small cases. Bases on these we have the following conjecture.

\begin{conjecture}
If $k+\ell+m+n$ is even, then we have $\emph{HARDY}(k,\ell,m,n)\ne 0$. Moreover, a necessary and sufficient condition that $\emph{HARDY}(k,\ell,m,n)< 0$ is that
\[
|i^k+i^\ell+i^m+i^n|=2.
\]
That is to say $\{k,\ell,m,n\}$ is the union of pairs, and these pairs satisfy, respectively, each of the conditions in \eqref{condition1}.
\end{conjecture}

Our method relies on the mollified second and mollified fourth moments of the Riemann zeta-function. Let $M(s)$ be the usual mollifier,
\[
M(s)=\sum_{n\leq y}\frac{\mu(n)P(\frac{\log y/n}{\log y})}{n^s},
\]
where $y=T^\vartheta$, $0<\vartheta<1$, and $P$ is some polynomial with $P(0)=0$ and $P(1)=1$. Trivially we have
\[
\int_T^{2T}Z^{(k)}(t)Z^{(\ell)}(t)|M(\tfrac12+it)|^2dt\leq \int_{S_{k,\ell}^+}Z^{(k)}(t)Z^{(\ell)}(t)|M(\tfrac12+it)|^2dt.
\]
An application of Cauchy-Schwarz's inequality then leads to
\begin{align*}
\int_T^{2T}Z^{(k)}(t)Z^{(\ell)}(t)&|M(\tfrac12+it)|^2dt\\
&\leq \text{meas}(S_{k,\ell}^+)^{1/2}\bigg(\int_T^{2T}Z^{(k)}(t)^2Z^{(\ell)}(t)^2|M(\tfrac12+it)|^4dt\bigg)^{1/2}.
\end{align*}
Thus, if the left hand side is non-negative, we can
square both sides and obtain
\begin{equation}\label{keyinequality+}
\text{meas}(S_{k,\ell}^+)\geq \frac{\mathcal{S}_{k,\ell}^2}{\mathcal{T}_{k,\ell}},
\end{equation}
where
\[
\mathcal{S}_{k,\ell}=\int_T^{2T}Z^{(k)}(t)Z^{(\ell)}(t)|M(\tfrac12+it)|^2dt
\]
and
\[
\mathcal{T}_{k,\ell}=\int_T^{2T}Z^{(k)}(t)^2Z^{(\ell)}(t)^2|M(\tfrac12+it)|^4dt.
\]
Similarly, if $\mathcal{S}_{k,\ell}\leq 0$, then we obtain that
\begin{equation}\label{keyinequality-}
\text{meas}(S_{k,\ell}^-)\geq \frac{\mathcal{S}_{k,\ell}^2}{\mathcal{T}_{k,\ell}}.
\end{equation}

The asymptotic formula for $\mathcal{S}_{k,\ell}$ follows from the work of Conrey [\textbf{\ref{C}}], while that for $\mathcal{T}_{k,\ell}$ follows from the shifted mollified fourth moment of the Riemann zeta-function which we shall prove in Proposition \ref{mollifiedfourth} in Section \ref{sectionfourth} below. Previously, only its correct order of magnitude and the asymptotic formula with $|M(\frac12+it)|^4$ being replaced by $|\widetilde{M}(\frac12+it)|^2$, where
\[
\widetilde{M}(s)=\sum_{n\leq y}\frac{a(n)P(\frac{\log y/n}{\log y})}{n^s}
\]
with $a(n)=\mu_2(n)$ or $a(n)=d_r(n)$, are known (see [\textbf{\ref{B}}; Theorem 1.4] and [\textbf{\ref{B1}}; Theorem 3.4] or [\textbf{\ref{BM}}; Lemma 3.1], respectively).

The paper is organised as follows. We include the necessary lemmas in Section \ref{sectionlemmas}. Section \ref{sectionSkl} is to evaluate $\mathcal{S}_{k,\ell}$. We obtain the asymptotic formula for the shifted mollified fourth moment of $\zeta(s)$ in Section \ref{sectionfourth} and from that derive the estimate for $\mathcal{T}_{k,\ell}$ in Section \ref{sectionTkl}. We deduce Theorem \ref{thm1} and Theorem \ref{thm2} in Section \ref{proofthm12}, and Theorem \ref{thmHardy} and Corollary \ref{corHardy} in Section \ref{proofthm3}.

\section{Various lemmas}\label{sectionlemmas}

\begin{lemma}\label{EM}
Suppose $f_j$ are fixed smooth functions for $1\leq j\leq 4$. Then we have
\begin{align*}
&\sum_{\substack{m_1n_1,m_2n_2\leq y\\m_1m_2,n_1n_2\leq y}}\frac{1}{m_1m_2n_1n_2}f_1\Big(\frac{\log y/m_1n_1}{\log y}\Big)f_2\Big(\frac{\log y/m_2n_2}{\log y}\Big)f_3\Big(\frac{\log y/m_1m_2}{\log y}\Big)f_4\Big(\frac{\log y/n_1n_2}{\log y}\Big)\\
&\qquad=(\log y)^4\iiiint_{\substack{t_1+t_3,t_2+t_4\leq 1\\t_1+t_2,t_3+t_4\leq 1}} f_1(1-t_1-t_3)f_2(1-t_2-t_4)\\
&\qquad\qquad\qquad\qquad\qquad\times f_3(1-t_1-t_2)f_4(1-t_3-t_4)dt_1dt_2dt_3dt_4+O\big((\log y)^3\big).
\end{align*}
\end{lemma}
\begin{proof}
We write
\begin{align*}
\sum_{\substack{m_1n_1,m_2n_2\leq y\\m_1m_2,n_1n_2\leq y}}&=\sum_{m_2\leq n_1\leq y}\sum_{m_1,n_2\leq y/n_1}+\sum_{n_1\leq m_2\leq y}\sum_{m_1,n_2\leq y/m_2}+O\big((\log y)^2\big)\\
&=A_1+A_2+O\big((\log y)^2\big),
\end{align*}
say.
By Lemma 4.4 in [\textbf{\ref{BCY}}] we have
\begin{align}\label{resultBCY}
\sum_{n\leq y_1}\frac{1}{n}f\Big(\frac{\log y_1/n}{\log y_1}\Big)g\Big(\frac{\log y_2/n}{\log y_2}\Big)=(\log y_1)\int_0^1f(1-t)g\Big(1-\frac{t\log y_1}{\log y_2}\Big)dt+O(1),
\end{align}
if $y_1\leq y_2$ and $f,g$ are smooth functions. So if $n_1\geq m_2$, then 
\begin{align*}
&\sum_{\substack{m_1\leq y/n_1}}\frac{1}{m_1}f_1\Big(\frac{\log y/m_1n_1}{\log y}\Big)f_3\Big(\frac{\log y/m_1m_2}{\log y}\Big)\\
&\qquad=\Big(\log \frac{y}{n_1}\Big)\int_0^1f_1\Big(\frac{(1-t)\log y/n_1}{\log y}\Big)f_3\Big(\frac{(1-t)\log y/n_1}{\log y}+\frac{\log n_1/m_2}{\log y}\Big)dt+O(1).
\end{align*}
A similar expression holds for the sum over $n_2$ and hence
\begin{align*}
&A_1=\int_0^1\int_0^1\sum_{n_1\leq y}\frac{1}{n_1}\Big(\log \frac{y}{n_1}\Big)^2f_1\Big(\frac{(1-t_1)\log y/n_1}{\log y}\Big)f_4\Big(\frac{(1-t_4)\log y/n_1}{\log y}\Big)\\
&\quad\times \sum_{m_2\leq n_1}\frac{1}{m_2}f_3\Big(\frac{(1-t_1)\log y/n_1}{\log y}+\frac{\log n_1/m_2}{\log y}\Big)f_2\Big(\frac{(1-t_4)\log y/n_1}{\log y}+\frac{\log n_1/m_2}{\log y}\Big)dt_1dt_4\\
&\quad+O\big((\log y)^3\big).
\end{align*}
An application of \eqref{resultBCY} to the above sum over $m_2$, followed by another one to the sum over $n_1$ leads to
\begin{align*}
A_1&=(\log y)^4\int_0^1\int_0^1\int_0^1\int_0^1t_3(1-t_3)^2f_1\big((1-t_1)(1-t_3)\big)f_4\big((1-t_4)(1-t_3)\big)\\
&\qquad\times f_3\big((1-t_1)(1-t_3)+(1-t_2)t_3\big)f_2\big((1-t_4)(1-t_3)+(1-t_2)t_3\big)dt_1dt_2dt_3dt_4\\
&\qquad+O\big((\log y)^3\big)\\
&=(\log y)^4\int_0^1\int_0^{1-t_3}\int_0^{t_3}\int_0^{1-t_3}f_1(1-t_1-t_3)f_4(1-t_3-t_4)\\
&\qquad\times f_3(1-t_1-t_2)f_2(1-t_2-t_4)dt_1dt_2dt_4dt_3+O\big((\log y)^3\big),
\end{align*}
after some changes of variables. We obtain a similar expression for $A_2$ and the result hence follows.
\end{proof}

\begin{lemma}
\label{logsave}
Suppose $-1 \leq \sigma \leq 0$.  Then
\begin{equation*}
\sum_{n \leq y} \frac{1}{n} \Big(\frac{y}{n}\Big)^{\sigma}  \ll \min\big\{|\sigma|^{-1}, \log y\big\}.
\end{equation*}
\end{lemma}
\begin{proof}
See [\textbf{\ref{BCY}}; Lemma 4.6], for instance.
\end{proof}

\section{Evaluate $\mathcal{S}_{k,\ell}$}\label{sectionSkl}

We need the following mollified second moment of the Riemann zeta-function [\textbf{\ref{C}}; Theorem 2] (see also [\textbf{\ref{CS}}; Theorem 5.1]).

\begin{proposition}\label{mollifiedsecond}
Let $P$, $Q_1$ and $Q_2$ be polynomials with $P(0)=0$. Then for any $\vartheta < 4/7$ we have
\begin{align*}
&\frac{1}{T}\int_{T}^{2T}Q_1\Big(\frac{1}{\log T}\frac{d}{d\alpha}\Big)Q_2\Big(\frac{1}{\log T}\frac{d}{d\beta}\Big)\zeta(\tfrac12+\alpha+it)\zeta(\tfrac12+\beta-it)|M(\tfrac12+it)|^2dt\bigg|_{\alpha=\beta=0}\\
&\qquad\qquad=Q_1(0)Q_2(0)P(1)^2+\frac{1}{\vartheta}\int_0^1\int_0^1\Big(Q_1(-t_1)P'(t_2)-\vartheta Q_1'(-t_1)P(t_2)\Big)\\
&\qquad\qquad\qquad\qquad\times\Big(Q_2(-t_1)P'(t_2)-\vartheta Q_2'(-t_1)P(t_2)\Big)dt_1dt_2+O\big((\log T)^{-1}\big).
\end{align*}
\end{proposition}

Note that
\begin{equation}\label{thetaprop}
\theta'(t)=\frac{\log T}{2}+O(1)\qquad\text{and}\qquad \theta^{(k)}(t)\ll_k T^{-k+1}
\end{equation}
for $t\asymp T$ and any $k\geq 2$. So
\begin{align}\label{Z'tformula}
Z^{(k)}(t)&=i^ke^{i\theta(t)}\sum_{j=0}^{k}\binom{k}{j}\Big(\frac{\log T}{2}\Big)^j\zeta^{k-j}(\tfrac12+it)+O\bigg(\sum_{j=0}^{k}(\log T)^{j-1}|\zeta^{k-j}(\tfrac12+it)|\bigg)\nonumber\\
&=i^ke^{i\theta(t)}(\log T)^kR_k\Big(\frac{1}{\log T}\frac{d}{d\alpha}\Big)\zeta(\tfrac{1}{2}+\alpha+it)\Big|_{\alpha=0}\\
&\qquad\qquad\qquad+O\bigg(\sum_{j=0}^{k}(\log T)^{j-1}|\zeta^{k-j}(\tfrac12+it)|\bigg),\nonumber
\end{align}
where 
\[
R_k(x)=\Big(\frac12+x\Big)^k.
\]
By Proposition \ref{mollifiedsecond} and Cauchy-Schwarz's inequality, the contribution of the $O$-terms to $\mathcal{S}_{k,\ell}$ is $O(T(\log T)^{k+\ell-1})$.
Hence, by noting that $Z^{(\ell)}(t)=(-1)^\ell Z^{(\ell)}(-t)$,
\begin{align*}
\mathcal{S}_{k,\ell}&=(-1)^\ell i^{k+\ell}(\log T)^{k+\ell}\int_{T}^{2T}R_k\Big(\frac{1}{\log T}\frac{d}{d\alpha}\Big)R_\ell\Big(\frac{1}{\log T}\frac{d}{d\beta}\Big)\\
&\qquad\qquad\zeta(\tfrac12+\alpha+it)\zeta(\tfrac12+\beta-it)|M(\tfrac12+it)|^2dt\bigg|_{\alpha=\beta=0}+O\big(T(\log T)^{k+\ell-1}\big).
\end{align*}
Using Proposition \ref{mollifiedsecond} again, the integral above is 
\begin{align*}
&T\bigg(\frac1{2^{k+\ell}}+\frac{1}{\vartheta}\int_0^1\int_{-1/2}^{1/2}\Big(t_1^kP'(t_2)-\vartheta k t_1^{k-1}P(t_2)\Big)\Big(t_1^\ell P'(t_2)-\vartheta \ell t_1^{\ell-1}P(t_2)\Big)dt_1dt_2\bigg)\\
&\qquad\qquad+O\big(T(\log T)^{-1}\big),
\end{align*}
by a change of variables $t_1\longrightarrow 1/2-t_1$. In the case $k+\ell=2s$ is even, this simplifies to
\begin{align*}
&T\bigg(\frac1{4^{s}}+\frac{2}{\vartheta}\int_0^1\int_{0}^{1/2}\Big(t_1^{2s}P'(t_2)^2+\vartheta^2 k\ell t_1^{2s-2}P(t_2)^2\Big)dt_1dt_2\bigg)+O\big(T(\log T)^{-1}\big)\\
&\qquad=\frac{T}{4^{s}}\bigg(1+\frac{1}{\vartheta(2s+1)}\int_0^1P'(t)^2dt+\frac{4\vartheta k\ell}{2s-1}\int_0^1P(t)^2dt\bigg)+O\big(T(\log T)^{-1}\big).
\end{align*}
Thus,
\begin{align}\label{asympS1}
\mathcal{S}_{k,\ell}&=\frac{(-1)^{\ell+s}}{4^{s}}T(\log T)^{2s}\bigg(1+\frac{1}{\vartheta(2s+1)}\int_0^1P'(t)^2dt+\frac{4\vartheta k\ell}{2s-1}\int_0^1P(t)^2dt\bigg)\nonumber\\
&\qquad\qquad+O\big(T(\log T)^{2s-1}\big).
\end{align}

\section{Shifted mollified fourth moment of the Riemann zeta-function}\label{sectionfourth}

We shall need the following twisted fourth moment of the Riemann zeta-function [\textbf{\ref{BBLR}}; Theorem 1.2]. Throughout this section, we let $w(t)$ be a smooth function with support in $[1,2]$ and satisfying $w^{(j)}(t)\ll_j T^\varepsilon$ for any $j\geq 0$.

\begin{theorem}[Bettin, Bui, Li and Radziwi\l\l]\label{BB}
Let $G(s)$ be an even entire function of rapid decay in any fixed strip $|\emph{Re}(s)|\leq C$ satisfying $G(0)=1$, and let
\begin{equation}\label{Vx}
W(x)=\frac{1}{2\pi i}\int_{(1)}G(s)(2\pi)^{-2s}x^{-s}\frac{ds}{s}.
\end{equation}
Then we have
\begin{align*}
&\sum_{m,n\leq y}\frac{a_m\overline{a_n}}{\sqrt{mn}}\int_{-\infty}^{\infty}\zeta(\tfrac{1}{2}+\alpha_1+it)\zeta(\tfrac{1}{2}+\alpha_2+it)\zeta(\tfrac{1}{2}+\beta_1-it)\zeta(\tfrac{1}{2}+\beta_2-it)\Big(\frac{m}{n}\Big)^{-it}w\Big(\frac{t}{T}\Big)dt\\
&\quad=\sum_{m,n\leq y}\frac{a_m\overline{a_n}}{\sqrt{mn}}\int_{-\infty}^{\infty}w\Big(\frac{t}{T}\Big)\bigg\{Z_{\alpha_1,\alpha_2,\beta_1,\beta_2,m,n}(t)+\Big(\frac{t}{2\pi}\Big)^{-(\alpha_1+\beta_1)}Z_{-\beta_1,\alpha_2,-\alpha_1,\beta_2,m,n}(t)\\
&\quad\quad+\Big(\frac{t}{2\pi}\Big)^{-(\alpha_1+\beta_2)}Z_{-\beta_2,\alpha_2,\beta_1,-\alpha_1,m,n}(t)+\Big(\frac{t}{2\pi}\Big)^{-(\alpha_2+\beta_1)}Z_{\alpha_1,-\beta_1,-\alpha_2,\beta_2,m,n}(t)\\
&\quad\quad+\Big(\frac{t}{2\pi}\Big)^{-(\alpha_2+\beta_2)}Z_{\alpha_1,-\beta_2,\beta_1,-\alpha_2,m,n}(t)+\Big(\frac{t}{2\pi}\Big)^{-(\alpha_1+\alpha_2+\beta_1+\beta_2)}Z_{-\beta_1,-\beta_2,-\alpha_1,-\alpha_2,m,n}(t)\bigg\}dt\\
&\quad\quad+O_\varepsilon(T^{1/2+2\vartheta+\varepsilon}+T^{3/4+\vartheta+\varepsilon})
\end{align*}
uniformly for $\alpha_1,\alpha_2,\beta_1,\beta_2\ll (\log T)^{-1}$, where 
\[
Z_{\alpha_1,\alpha_2,\beta_1,\beta_2,m,n}(t)=\sum_{ma_1a_2=nb_1b_2}\frac{1}{a_{1}^{1/2+\alpha_1}a_{2}^{1/2+\alpha_2}b_{1}^{1/2+\beta_1}b_{2}^{1/2+\beta_2}}W\Big(\frac{a_1a_2b_1b_2}{t^2}\Big).
\]
\end{theorem}

\subsection{Shifted mollified fourth moment}

Using Theorem \ref{BB} we can prove the following shifted mollified fourth moment of the Riemann zeta-function.

\begin{proposition}\label{mollifiedfourth}
Let $P$ be a polynomial with $P(0)=P'(0)=0$. Then for any $\vartheta <1/8$ we have
 \begin{align*}
&\int_{T}^{2T}\zeta(\tfrac12+\alpha_1+it)\zeta(\tfrac12+\alpha_2+it)\zeta(\tfrac12+\beta_1-it)\zeta(\tfrac12+\beta_2-it)|M(\tfrac12+it)|^4dt\\
&\quad=\frac{T}{2\vartheta^4}\frac{d^8}{\prod_{1\leq j\leq 4}dx_jdz_j}\int_0^1\int_0^1\int_0^1\int_0^1 \big(1+\vartheta\sum x_j\big)\big(1+\vartheta\sum z_j\big)\nonumber\\
&\qquad\times\Big(u_1-u_2+\vartheta\big(-x_1-x_2+z_1+z_2+u_1\sum x_j-u_2\sum z_j\big)\Big)\\
&\qquad\times\Big(u_1-u_2+\vartheta\big(-x_3-x_4+z_3+z_4+u_1\sum x_j-u_2\sum z_j\big)\Big)\\
&\qquad\times y^{\alpha_1(x_1+x_2)+\alpha_2(z_1+z_2)+\beta_1(x_3+x_4)+\beta_2(z_3+z_4)}(Ty^{\sum x_j})^{-(\alpha_1+\beta_1)u_1}(Ty^{\sum z_j})^{-(\alpha_2+\beta_2)u_2}\\
&\qquad\times \Big(T^{u_1-u_2}y^{-x_1-x_2+z_1+z_2+u_1\sum x_j-u_2\sum z_j}\Big)^{(\alpha_1-\alpha_2)u_3}\\
&\qquad\times\Big(T^{u_1-u_2}y^{-x_3-x_4+z_3+z_4+u_1\sum x_j-u_2\sum z_j}\Big)^{(\beta_1-\beta_2)u_4}\\
&\qquad\times\iiiint_{\substack{t_1+t_3,t_2+t_4\leq 1\\t_1+t_2,t_3+t_4\leq 1}} P(1-t_1-t_3+x_1+z_1)P(1-t_2-t_4+x_2+z_2)\\
&\qquad\times P(1-t_1-t_2+x_3+z_3)P(1-t_3-t_4+x_4+z_4)dt_1dt_2dt_3dt_4du_1du_2du_3du_4\bigg|_{\underline{x}=\underline{z}=\underline{0}}\nonumber\\
&\qquad+O_\varepsilon\big(T(\log T)^{-1+\varepsilon}\big)
 \end{align*}
uniformly for $\alpha_1,\alpha_2,\beta_1,\beta_2\ll (\log T)^{-1}$.
\end{proposition}
\begin{proof}
We shall establish the result for the smoothed version of the shifted mollified fourth moment,
\begin{align*}
&I(\alpha_1,\alpha_2,\beta_1,\beta_2)\\
&\quad=\int_{-\infty}^{\infty}\zeta(\tfrac12+\alpha_1+it)\zeta(\tfrac12+\alpha_2+it)\zeta(\tfrac12+\beta_1-it)\zeta(\tfrac12+\beta_2-it)|M(\tfrac12+it)|^4w\Big(\frac{t}{T}\Big)dt.
\end{align*}
It is a standard exercise to deduce our proposition from that.

 Due to the holomorphy in $\alpha_i$ and $\beta_j$, it suffices by the maximum modulus principle to prove the proposition uniformly over any fixed annuli such that $|\alpha_i|,|\beta_j| \asymp (\log T)^{-1}$, $|\alpha_i + \beta_j| \gg (\log T)^{-1}$ for any $1\leq i,j\leq 2$.

It is convenient to prescribe certain conditions on the function $G(s)$ in Theorem \ref{BB}. To be precise, we assume that $G(s)$ vanishes at $s=-\frac{(\alpha_i+\beta_j)}{2}$ for $1\leq i,j\leq 2$. Provided that $\vartheta<1/8$, we write
\begin{equation}\label{firstestimate1}
I(\alpha_1,\alpha_2,\beta_1,\beta_2)=I_1+I_2+I_3+I_4+I_5+I_6+O_\varepsilon(T^{1-\varepsilon})
\end{equation}
correspondingly to the decomposition in Theorem \ref{BB}. We first consider $I_1$, which is equal to
\begin{align*}
&\sum_{m_1,m_2,n_1,n_2\leq y}\frac{\mu_(m_1)\mu(m_2)\mu(n_1)\mu(n_2)P\big(\frac{\log y/m_1}{\log y}\big)P\big(\frac{\log y/m_2}{\log y}\big) P\big(\frac{\log y/n_1}{\log y}\big)P\big(\frac{\log y/n_2}{\log y}\big) }{\sqrt{m_1m_2n_1n_2}} \\
&\qquad\qquad\times\sum_{m_1m_2a_1a_2=n_1n_2b_1b_2}\frac{1}{a_1^{1/2+\alpha_1}a_2^{1/2+\alpha_2}b_1^{1/2+\beta_1}b_2^{1/2+\beta_2}}\int_{-\infty}^{\infty}W\Big(\frac{a_1a_2b_1b_2}{t^2}\Big)w\Big(\frac{t}{T}\Big)dt.
\end{align*}
Note that for $P(x)=\sum_{j\geq2}c_jx^j$, we can write
\[
P\Big(\frac{\log y/n}{\log y}\Big)=\sum_{j\geq 2}\frac{c_jj!}{(\log y)^j}\frac{1}{2\pi i}\int_{(1)}\Big(\frac{y}{n}\Big)^u\frac{du}{u^{j+1}}.
\]
Combining with \eqref{Vx} we get
\begin{align}\label{Jformula}
&I_1=\sum_{k_1,k_2,\ell_1,\ell_2\geq2}\frac{c_{k_1}c_{k_2}c_{\ell_1}c_{\ell_2} k_1!k_2!\ell_1!\ell_2!}{(\log y)^{k_1+k_2+\ell_1+\ell_2}}\int_{-\infty}^{\infty}w\Big(\frac{t}{T}\Big)\frac{1}{(2\pi i)^5}\int_{(1)^5}G(s)\Big(\frac{t}{2\pi}\Big)^{2s}y^{u_1+u_2+v_1+v_2}\nonumber\\
&\quad\times\sum_{{m_1m_2a_1a_2=n_1n_2b_1b_2}}\frac{\mu_(m_1)\mu(m_2)\mu(n_1)\mu(n_2)}{m_1^{1/2+u_1}m_2^{1/2+u_2}n_1^{1/2+v_1}n_2^{1/2+v_2}a_1^{1/2+\alpha_1+s}a_2^{1/2+\alpha_2+s}b_1^{1/2+\beta_1+s}b_2^{1/2+\beta_2+s}}\nonumber\\
&\quad\times \frac{du_1}{u_1^{k_1+1}}\frac{du_2}{u_2^{k_2+1}}\frac{dv_1}{v_1^{\ell_1+1}}\frac{dv_2}{v_2^{\ell_2+1}}\frac{ds}{s}dt.
\end{align}
The arithmetical sum is
\begin{align}\label{I355}
&\sum_{{m_1m_2a_1a_2=n_1n_2b_1b_2}}\frac{\mu_(m_1)\mu(m_2)\mu(n_1)\mu(n_2)}{m_1^{1/2+u_1}m_2^{1/2+u_2}n_1^{1/2+v_1}n_2^{1/2+v_2}a_1^{1/2+\alpha_1+s}a_2^{1/2+\alpha_2+s}b_1^{1/2+\beta_1+s}b_2^{1/2+\beta_2+s}}\\
&\qquad=A(\alpha_1,\alpha_2,\beta_1,\beta_2,u_1,u_2,v_1,v_2,s)\frac{\prod_{1\leq i,j\leq 2}\zeta(1+\alpha_i+\beta_j+2s)\zeta(1+u_i+v_j)}{\prod_{1\leq i,j\leq 2}\zeta(1+\alpha_i+v_j+s)\zeta(1+\beta_i+u_j+s)},\nonumber
\end{align}
where $A(\alpha_1,\alpha_2,\beta_1,\beta_2,u_1,u_2,v_1,v_2,s)$ is an arithmetical factor converging absolutely in a product of half-planes containing the origin. 

We first move the $u_i$ and $v_j$ contours, $1\leq i,j\leq 2$, in \eqref{Jformula} to $\textrm{Re}(u_i)=\textrm{Re}(v_j)=\delta$, and then move the $s$ contour to $\textrm{Re}(s)=-\delta/2$, where
$\delta> 0$ is some fixed constant such that the arithmetical factor converges absolutely. In doing so, we only cross a simple pole at $s=0$. Note that the poles at $s=-\frac{(\alpha_i+\beta_j)}{2}$, $1\leq i,j\leq 2$, of the zeta-functions are cancelled out by the zeros of $G(s)$. On the new line we simply bound the integral by absolute values, giving the following contribution
\[
\ll_\varepsilon T^{1-\delta+\varepsilon}y^{4\delta}\ll_\varepsilon T^{1-\varepsilon}.
\]
Hence
\begin{align}\label{I340}
I_1&=\widehat{w}(0)T\prod_{1\leq i,j\leq 2}\zeta(1+\alpha_i+\beta_j)\sum_{k_1,k_2,\ell_1,\ell_2\geq2}\frac{c_{k_1}c_{k_2}c_{\ell_1}c_{\ell_2} k_1!k_2!\ell_1!\ell_2!}{(\log y)^{k_1+k_2+\ell_1+\ell_2}}\nonumber\\
&\qquad\times \frac{1}{(2\pi i)^4}\int_{(\delta)^4}y^{u_1+u_2+v_1+v_2}A(\alpha_1,\alpha_2,\beta_1,\beta_2,u_1,u_2,v_1,v_2,0)\\
&\qquad\qquad\times\frac{\prod_{1\leq i,j\leq 2}\zeta(1+u_i+v_j)}{\prod_{1\leq i,j\leq 2}\zeta(1+\alpha_i+v_j)\zeta(1+\beta_i+u_j)}\frac{du_1}{u_1^{k_1+1}}\frac{du_2}{u_2^{k_2+1}}\frac{dv_1}{v_1^{\ell_1+1}}\frac{dv_2}{v_2^{\ell_2+1}}+O_\varepsilon \big(T^{1-\varepsilon}\big).\nonumber
\end{align}

We now move the contours of integration to $\textrm{Re}(u_i)\asymp(\log T)^{-1}$, $\textrm{Re}(v_j)\asymp (\log T)^{-1}$, $1\leq i,j\leq 2$. Bounding the integrals trivially shows that $I_1\ll T$. Hence we can replace $A(\alpha_1,\alpha_2,\beta_1,\beta_2,u_1,u_2,v_1,v_2,0)$ by $A(\underline{0})$ in \eqref{I340} with an error of size $O(T(\log T)^{-1})$. By letting $\alpha_i=\beta_j=0$ and $u_i=v_j=s$, $1\leq i,j\leq 2$, in \eqref{I355} we get
\[
A(0,0,0,0,s,s,s,s,s)=\sum_{{m_1m_2a_1a_2=n_1n_2b_1b_2}}\frac{\mu_(m_1)\mu(m_2)\mu(n_1)\mu(n_2)}{(m_1m_2n_1n_2a_1a_2b_1b_2)^{1/2+s}}=1,
\]
and so $A(\underline{0})=1$, in particular.

We next use the Dirichlet series for $\zeta(1+u_i+v_j)$, $1\leq i,j\leq 2$, and reverse the order of summation and integration. The $u_i$ and $v_j$ variables in \eqref{I340} are now separated so that
\begin{align}\label{I356}
I_1&=\widehat{w}(0)T\prod_{1\leq i,j\leq 2}\zeta(1+\alpha_i+\beta_j)\sum_{\substack{m_1n_1,m_2n_2\leq y\\m_1m_2,n_1n_2\leq y}}\frac{1}{m_1m_2n_1n_2}\nonumber\\
&\qquad\times K_{\alpha_1,\alpha_2}(m_1n_1) K_{\alpha_1,\alpha_2}(m_2n_2) K_{\beta_1,\beta_2}(m_1m_2) K_{\beta_1,\beta_2}(n_1n_2)+O\big(T(\log T)^{-1}\big),
\end{align}
where
\[
K_{\alpha_1,\alpha_2}(n)=\sum_{j\geq 2}\frac{c_jj!}{(\log y)^j}\frac{1}{2\pi i}\int_{(\asymp (\log T)^{-1})}\Big(\frac{y}{n}\Big)^{u}\frac{1}{\zeta(1+\alpha_1+u)\zeta(1+\alpha_2+u)}\frac{du}{u^{j+1}}.
\]
Here we are able to restrict the sum over $m_i,n_j$ to $m_1n_1,m_2n_2,m_1m_2,n_1n_2\leq y$ by moving the $u_i,v_j$-integrals far to the right otherwise.  

The expression $K_{\alpha_1,\alpha_2}(n)$ is evaluated in [\textbf{\ref{BCY}}; Lemma 5.7, Lemma 6.1],
\begin{align*}
K_{\alpha_1,\alpha_2}(n)&=\frac{1}{(\log y)^2}\frac{d^2}{dx_1dz_1}y^{\alpha_1x_1+\alpha_2z_1}P\Big(x_1+z_1+\frac{\log y/n}{\log y}\Big)\bigg|_{x_1=z_1=0}\\
&\qquad\qquad+O\big((\log T)^{-3}\big)+O_\varepsilon\bigg(\Big(\frac{y}{n}\Big)^{-\nu}(\log T)^{-2+\varepsilon}\bigg),
\end{align*}
for some $\nu\asymp(\log\log y)^{-1}$. Putting this into \eqref{I356} we get
\begin{align*}
I_1&=\frac{\widehat{w}(0)T}{(\log y)^8}\prod_{1\leq i,j\leq 2}\zeta(1+\alpha_i+\beta_j)\frac{d^8}{\prod_{1\leq j\leq 4}dx_jdz_j}y^{\alpha_1(x_1+x_2)+\alpha_2(z_1+z_2)+\beta_1(x_3+x_4)+\beta_2(z_3+z_4)}
\\
&\qquad\times\sum_{\substack{m_1n_1,m_2n_2\leq y\\m_1m_2,n_1n_2\leq y}}\frac{1}{m_1m_2n_1n_2}P\Big(x_1+z_1+\frac{\log y/m_1n_1}{\log y}\Big)P\Big(x_2+z_2+\frac{\log y/m_2n_2}{\log y}\Big)\\
&\qquad\times P\Big(x_3+z_3+\frac{\log y/m_1m_2}{\log y}\Big)P\Big(x_4+z_4+\frac{\log y/n_1n_2}{\log y}\Big)\bigg|_{\underline{x}=\underline{z}=\underline{0}}+O_\varepsilon\big(T(\log T)^{-1+\varepsilon
}\big),
 \end{align*}
 by Lemma \ref{logsave}. Using Lemma \ref{EM} we then obtain that
 \begin{align}\label{secondestimate1}
 I_1&=\frac{\widehat{w}(0)T}{(\log y)^4}\prod_{1\leq i,j\leq 2}\zeta(1+\alpha_i+\beta_j)\frac{d^8}{\prod_{1\leq j\leq 4}dx_jdz_j}y^{\alpha_1(x_1+x_2)+\alpha_2(z_1+z_2)+\beta_1(x_3+x_4)+\beta_2(z_3+z_4)}\nonumber
\\
&\qquad\times\iiiint_{\substack{t_1+t_3,t_2+t_4\leq 1\\t_1+t_2,t_3+t_4\leq 1}} P(1-t_1-t_3+x_1+z_1)P(1-t_2-t_4+x_2+z_2)\\
&\qquad\qquad\qquad\times P(1-t_1-t_2+x_3+z_3)P(1-t_3-t_4+x_4+z_4)dt_1dt_2dt_3dt_4\bigg|_{\underline{x}=\underline{z}=\underline{0}}\nonumber\\
&\qquad+O_\varepsilon\big(T(\log T)^{-1+\varepsilon
}\big).\nonumber
 \end{align}
 
 Note that $I_2$ is essentially obtained by multiplying $I_1$ with $T^{-(\alpha_1+\beta_1)}$ and changing the shifts $\alpha_1\leftrightarrow-\beta_1$, $I_3$ is obtained by multiplying $I_1$ with $T^{-(\alpha_1+\beta_2)}$ and changing the shifts $\alpha_1\leftrightarrow-\beta_2$, $I_4$ is obtained by multiplying $I_1$ with $T^{-(\alpha_2+\beta_1)}$ and changing the shifts $\alpha_2\leftrightarrow-\beta_1$, $I_5$ is obtained by multiplying $I_1$ with $T^{-(\alpha_2+\beta_2)}$ and changing the shifts $\alpha_2\leftrightarrow-\beta_2$, and $I_6$ is obtained by multiplying $I_1$ with $T^{-(\alpha_1+\alpha_2+\beta_1+\beta_2)}$ and changing the shifts $\alpha_1\leftrightarrow-\beta_1$ and $\alpha_2\leftrightarrow-\beta_2$. Hence from \eqref{firstestimate1} and \eqref{secondestimate1} we obtain that
\begin{align}\label{finalI}
&I(\alpha_1,\alpha_2,\beta_1,\beta_2)=\frac{\widehat{w}(0)T}{(\log y)^4}\frac{d^8}{\prod_{1\leq j\leq 4}dx_jdz_j}U(\underline{x},\underline{z})\nonumber
\\
&\qquad\times\iiiint_{\substack{t_1+t_3,t_1+t_4\leq 1\\t_2+t_3,t_2+t_4\leq 1}} P(1-t_1-t_3+x_1+z_1)P(1-t_1-t_4+x_2+z_2)\\
&\qquad\qquad\qquad\quad\times P(1-t_2-t_3+x_3+z_3)P(1-t_2-t_4+x_4+z_4)dt_1dt_2dt_3dt_4\bigg|_{\underline{x}=\underline{z}=\underline{0}}\nonumber\\
&\qquad+O_\varepsilon\big(T(\log T)^{-1+\varepsilon
}\big),\nonumber
\end{align}
where
\begin{align*}
U(\underline{x},\underline{z})&=\frac{y^{\alpha_1(x_1+x_2)+\alpha_2(z_1+z_2)+\beta_1(x_3+x_4)+\beta_2(z_3+z_4)}}{(\alpha_1+\beta_1)(\alpha_1+\beta_2)(\alpha_2+\beta_1)(\alpha_2+\beta_2)}\\
&\qquad\qquad -\frac{T^{-(\alpha_1+\beta_1)}y^{-\beta_1(x_1+x_2)+\alpha_2(z_1+z_2)-\alpha_1(x_3+x_4)+\beta_2(z_3+z_4)}}{(\alpha_1+\beta_1)(-\beta_1+\beta_2)(\alpha_2-\alpha_1)(\alpha_2+\beta_2)}\\
&\qquad\qquad-\frac{T^{-(\alpha_1+\beta_2)}y^{-\beta_2(x_1+x_2)+\alpha_2(z_1+z_2)+\beta_1(x_3+x_4)-\alpha_1(z_3+z_4)}}{(-\beta_2+\beta_1)(\alpha_1+\beta_2)(\alpha_2+\beta_1)(\alpha_2-\alpha_1)}\\
&\qquad\qquad-\frac{T^{-(\alpha_2+\beta_1)}y^{\alpha_1(x_1+x_2)-\beta_1(z_1+z_2)-\alpha_2(x_3+x_4)+\beta_2(z_3+z_4)}}{(\alpha_1-\alpha_2)(\alpha_1+\beta_2)(\alpha_2+\beta_1)(-\beta_1+\beta_2)}\\
&\qquad\qquad-\frac{T^{-(\alpha_2+\beta_2)}y^{\alpha_1(x_1+x_2)-\beta_2(z_1+z_2)+\beta_1(x_3+x_4)-\alpha_2(z_3+z_4)}}{(\alpha_1+\beta_1)(\alpha_1-\alpha_2)(-\beta_2+\beta_1)(\alpha_2+\beta_2)}\\
&\qquad\qquad+\frac{T^{-(\alpha_1+\alpha_2+\beta_1+\beta_2)}y^{-\beta_1(x_1+x_2)-\beta_2(z_1+z_2)-\alpha_1(x_3+x_4)-\alpha_2(z_3+z_4)}}{(\alpha_1+\beta_1)(\beta_1+\alpha_2)(\beta_2+\alpha_1)(\alpha_2+\beta_2)}.
\end{align*}

We write
\begin{eqnarray*}
\frac{y^{\alpha_1(x_1+x_2)+\alpha_2(z_1+z_2)+\beta_1(x_3+x_4)+\beta_2(z_3+z_4)}}{(\alpha_1+\beta_1)(\alpha_1+\beta_2)(\alpha_2+\beta_1)(\alpha_2+\beta_2)}&=&\frac{y^{\alpha_1(x_1+x_2)+\alpha_2(z_1+z_2)+\beta_1(x_3+x_4)+\beta_2(z_3+z_4)}}{(\alpha_1+\beta_1)(-\beta_1+\beta_2)(\alpha_2-\alpha_1)(\alpha_2+\beta_2)}\\
&&\quad+\frac{y^{\alpha_1(x_1+x_2)+\alpha_2(z_1+z_2)+\beta_1(x_3+x_4)+\beta_2(z_3+z_4)}}{(-\beta_2+\beta_1)(\alpha_1+\beta_2)(\alpha_2+\beta_1)(\alpha_2-\alpha_1)}
\end{eqnarray*}
and
\begin{align}\label{swap1}
&\frac{T^{-(\alpha_1+\alpha_2+\beta_1+\beta_2)}y^{-\beta_1(x_1+x_2)-\beta_2(z_1+z_2)-\alpha_1(x_3+x_4)-\alpha_2(z_3+z_4)}}{(\alpha_1+\beta_1)(\beta_1+\alpha_2)(\beta_2+\alpha_1)(\alpha_2+\beta_2)}\nonumber\\
&\qquad\quad=\frac{T^{-(\alpha_1+\alpha_2+\beta_1+\beta_2)}y^{-\beta_1(x_1+x_2)-\beta_2(z_1+z_2)-\alpha_1(x_3+x_4)-\alpha_2(z_3+z_4)}}{(\alpha_1+\beta_1)(-\beta_1+\beta_2)(\alpha_2-\alpha_1)(\alpha_2+\beta_2)}\nonumber\\
&\qquad\qquad\qquad+\frac{T^{-(\alpha_1+\alpha_2+\beta_1+\beta_2)}y^{-\beta_1(x_1+x_2)-\beta_2(z_1+z_2)-\alpha_1(x_3+x_4)-\alpha_2(z_3+z_4)}}{(-\beta_2+\beta_1)(\alpha_1+\beta_2)(\alpha_2+\beta_1)(\alpha_2-\alpha_1)}.
\end{align}
Note that we can change the roles $x_j\leftrightarrow z_j$ for any $1\leq j\leq 4$ in any term of $U(\underline{x},\underline{z})$ without affecting the value of $I(\alpha_1,\alpha_2,\beta_1,\beta_2)$ in \eqref{finalI}. Applying all these changes to the last term in \eqref{swap1}, we can replace $U(\underline{x},\underline{z})$ with
\begin{eqnarray}\label{combin}
&&\frac{y^{\alpha_1(x_1+x_2)+\alpha_2(z_1+z_2)+\beta_1(x_3+x_4)+\beta_2(z_3+z_4)}}{(-\beta_1+\beta_2)(\alpha_2-\alpha_1)}\bigg(\frac{1-(Ty^{\sum x_j})^{-(\alpha_1+\beta_1)}}{\alpha_1+\beta_1}\bigg)\bigg(\frac{1-(Ty^{\sum z_j})^{-(\alpha_2+\beta_2)}}{\alpha_2+\beta_2}\bigg)\nonumber\\
&&\qquad-\frac{y^{\alpha_1(x_1+x_2)+\alpha_2(z_1+z_2)+\beta_1(x_3+x_4)+\beta_2(z_3+z_4)}}{(-\beta_1+\beta_2)(\alpha_2-\alpha_1)}\bigg(\frac{1-(Ty^{x_1+x_2+z_3+z_4})^{-(\alpha_1+\beta_2)}}{\alpha_1+\beta_2}\bigg)\\
&&\qquad\qquad\qquad\times\bigg(\frac{1-(Ty^{z_1+z_2+x_3+x_4})^{-(\alpha_2+\beta_1)}}{\alpha_2+\beta_1}\bigg).\nonumber
\end{eqnarray}
Using the integral formula
\begin{equation}\label{int}
\frac{1-y^{-(\alpha+\beta)}}{\alpha+\beta}=(\log y)\int_{0}^{1}y^{-(\alpha+\beta)u}du
\end{equation} we then get
\begin{align}\label{finalformulaI}
&I(\alpha_1,\alpha_2,\beta_1,\beta_2)=\frac{\widehat{w}(0)T}{\vartheta^2(\log y)^2}\frac{d^8}{\prod_{1\leq j\leq 4}dx_jdz_j}\int_0^1\int_0^1\frac{U_1(\underline{x},\underline{z},u_1,u_2)-U_2(\underline{x},\underline{z},u_1,u_2)}{(-\beta_1+\beta_2)(\alpha_2-\alpha_1)}\nonumber\\
&\qquad\times\iiiint_{\substack{t_1+t_3,t_1+t_4\leq 1\\t_2+t_3,t_2+t_4\leq 1}} P(1-t_1-t_3+x_1+z_1)P(1-t_1-t_4+x_2+z_2)\\
&\qquad\qquad\quad\times P(1-t_2-t_3+x_3+z_3)P(1-t_2-t_4+x_4+z_4)dt_1dt_2dt_3dt_4du_1du_2\bigg|_{\underline{x}=\underline{z}=\underline{0}}\nonumber\\
&\qquad+O_\varepsilon\big(T(\log T)^{-1+\varepsilon
}\big),\nonumber
\end{align}
where
\begin{align*}
U_1&=y^{\alpha_1(x_1+x_2)+\alpha_2(z_1+z_2)+\beta_1(x_3+x_4)+\beta_2(z_3+z_4)}(Ty^{\sum x_j})^{-(\alpha_1+\beta_1)u_1}(Ty^{\sum z_j})^{-(\alpha_2+\beta_2)u_2}\\
&\qquad\qquad\times\big(1+\vartheta\sum x_j\big)\big(1+\vartheta\sum z_j\big)
\end{align*}
and
\begin{align*}
&U_2= y^{\alpha_1(x_1+x_2)+\alpha_2(z_1+z_2)+\beta_1(x_3+x_4)+\beta_2(z_3+z_4)}(Ty^{x_1+x_2+z_3+z_4})^{-(\alpha_1+\beta_2)u_1}\\
&\ \qquad\times(Ty^{z_1+z_2+x_3+x_4})^{-(\alpha_2+\beta_1)u_2}\big(1+\vartheta(x_1+x_2+z_3+z_4)\big)\big(1+\vartheta(z_1+z_2+x_3+x_4)\big).
\end{align*}

Again note that $I(\alpha_1,\alpha_2,\beta_1,\beta_2)$ is unchanged if we swap any of the pairs of variables $x_j\leftrightarrow z_j$ for any $1\leq j\leq 4$ or $u_1\leftrightarrow u_2$ in $U_1$ or $U_2$. Hence we can replace the term $U_1-U_2$ in the integrand with
\begin{eqnarray*}
&&\frac{1}{2}\Big(U_1(\underline{x},\underline{z},u_1,u_2)-U_2(z_1,z_2,x_3,x_4,x_1,x_2,z_3,z_4,u_2,u_1)\\
&&\qquad\qquad-U_2(x_1,x_2,z_3,z_4,z_1,z_2,x_3,x_4,u_1,u_2)+U_1(\underline{z},\underline{x},u_2,u_1)\Big),
\end{eqnarray*}
which is
\begin{align*}
&\frac{\big(1+\vartheta\sum x_j\big)\big(1+\vartheta\sum z_j\big)}{2}\\
&\qquad\times\bigg(y^{\alpha_1(x_1+x_2)+\alpha_2(z_1+z_2)+\beta_1(x_3+x_4)+\beta_2(z_3+z_4)}(Ty^{\sum x_j})^{-(\alpha_1+\beta_1)u_1}(Ty^{\sum z_j})^{-(\alpha_2+\beta_2)u_2}\\
&\qquad\qquad\ \  -y^{\alpha_1(z_1+z_2)+\alpha_2(x_1+x_2)+\beta_1(x_3+x_4)+\beta_2(z_3+z_4)}(Ty^{\sum z_j})^{-(\alpha_1+\beta_2)u_2}(Ty^{\sum x_j})^{-(\alpha_2+\beta_1)u_1}\\
&\qquad\qquad\ \ - y^{\alpha_1(x_1+x_2)+\alpha_2(z_1+z_2)+\beta_1(z_3+z_4)+\beta_2(x_3+x_4)}(Ty^{\sum x_j})^{-(\alpha_1+\beta_2)u_1}(Ty^{\sum z_j})^{-(\alpha_2+\beta_1)u_2}\\
&\qquad\qquad\ \ +y^{\alpha_1(z_1+z_2)+\alpha_2(x_1+x_2)+\beta_1(z_3+z_4)+\beta_2(x_3+x_4)}(Ty^{\sum z_j})^{-(\alpha_1+\beta_1)u_2}(Ty^{\sum x_j})^{-(\alpha_2+\beta_2)u_1}\bigg)\\
&\ =\frac{\big(1+\vartheta\sum x_j\big)\big(1+\vartheta\sum z_j\big)}{2}y^{\alpha_1(x_1+x_2)+\alpha_2(z_1+z_2)+\beta_1(x_3+x_4)+\beta_2(z_3+z_4)}(Ty^{\sum x_j})^{-(\alpha_1+\beta_1)u_1}\\
&\qquad\times(Ty^{\sum z_j})^{-(\alpha_2+\beta_2)u_2}\bigg(1-\Big(T^{u_1-u_2}y^{-x_1-x_2+z_1+z_2+u_1\sum x_j-u_2\sum z_j}\Big)^{\alpha_1-\alpha_2}\bigg)\\
&\qquad\times\bigg(1-\Big(T^{u_1-u_2}y^{-x_3-x_4+z_3+z_4+u_1\sum x_j-u_2\sum z_j}\Big)^{\beta_1-\beta_2}\bigg).
\end{align*}
Using \eqref{int} again in \eqref{finalformulaI} and simplifying we obtain the proposition.
\end{proof}

\subsection{Shifted fourth moment}

The shifted fourth moment of $\zeta(s)$ is standard. The derivation is similar and easier to the above subsection so we omit the proof. 

\begin{proposition}\label{shiftedfourth}
Let
\begin{align*}
J(\alpha_1,\alpha_2,\beta_1,\beta_2)=\int_{T}^{2T}\zeta(\tfrac12+\alpha_1+it)\zeta(\tfrac12+\alpha_2+it)\zeta(\tfrac12+\beta_1-it)\zeta(\tfrac12+\beta_2-it)dt.
\end{align*}
Then we have
\begin{align*}
&J(\alpha_1,\alpha_2,\beta_1,\beta_2)=\frac{T(\log T)^4}{2\zeta(2)}\int_0^1\int_0^1\int_0^1\int_0^1(u_1-u_2)^2\\
&\ \times T^{-(\alpha_1+\beta_1)u_1}T^{-(\alpha_2+\beta_2)u_2}T^{(\alpha_1-\alpha_2)(u_1-u_2)u_3}T^{(\beta_1-\beta_2)(u_1-u_2)u_4}du_1du_2du_3du_4 +O\big(T(\log T)^3\big)
\end{align*}
uniformly for $\alpha_1,\alpha_2,\beta_1,\beta_2\ll (\log T)^{-1}$.
\end{proposition}

\section{Evaluate $\mathcal{T}_{k,\ell}$}\label{sectionTkl}

We recall from \eqref{Z'tformula} that
\begin{align*}
Z^{(k)}(t)&=i^ke^{i\theta(t)}(\log T)^kR_k\Big(\frac{1}{\log T}\frac{d}{d\alpha}\Big)\zeta(\tfrac{1}{2}+\alpha+it)\Big|_{\alpha=0}\\
&\qquad\qquad\qquad+O\bigg(\sum_{j=0}^{k}(\log T)^{j-1}|\zeta^{k-j}(\tfrac12+it)|\bigg),\nonumber
\end{align*}
and the fact that $Z^{(\ell)}(t)=(-1)^\ell Z^{(\ell)}(-t)$. By Proposition \ref{mollifiedfourth} and Cauchy-Schwarz's inequality we hence obtain that
\begin{equation}\label{generalboundTkl}
\mathcal{T}_{k,\ell}\ll T(\log T)^{2(k+\ell)},
\end{equation}
and, in particular,
\begin{align*}
\mathcal{T}_{0,2}&=(\log T)^{4}R_2\Big(\frac{1}{\log T}\frac{d}{d\alpha_2}\Big)R_2\Big(\frac{1}{\log T}\frac{d}{d\beta_2}\Big)I(0,\alpha_2,0,\beta_2)\bigg|_{\alpha_2=\beta_2=0}+O_\varepsilon\big(T(\log T)^{3+\varepsilon}\big).
\end{align*}

Since $I(\alpha_1,\alpha_2,\beta_1,\beta_2)$ is holomorphic with respect to $\alpha_i, \beta_j$ small, the derivatives in the above expression can be obtained as integrals of radii $\asymp (\log T)^{-1}$ around the points $\alpha_i=\beta_j=0$, using Cauchy's integral formula.  Since the error term holds uniformly on these contours, the error term that holds for $I(\alpha_1,\alpha_2,\beta_1,\beta_2)$ also holds for its derivatives.

Note that
\begin{equation*}
\label{eq:Qop}
R\Big(\frac{1}{\log T}\frac{d}{d\alpha}\Big)X^{\alpha}=R\Big(\frac{\log X}{\log T}\Big)X^{\alpha}
\end{equation*}
for any polynomial $R$. Hence from Proposition \ref{mollifiedfourth} we get
 \begin{align}\label{asympS2}
\mathcal{T}_{0,2}&=\frac{T(\log T)^4}{2\vartheta^4}\frac{d^8}{\prod_{1\leq j\leq 4}dx_jdz_j}\int_0^1\int_0^1\int_0^1\int_0^1 \big(1+\vartheta\sum x_j\big)\big(1+\vartheta\sum z_j\big)\nonumber\\
&\quad\times\Big(u_1-u_2+\vartheta\big(-x_1-x_2+z_1+z_2+u_1\sum x_j-u_2\sum z_j\big)\Big)\nonumber\\
&\quad\times\Big(u_1-u_2+\vartheta\big(-x_3-x_4+z_3+z_4+u_1\sum x_j-u_2\sum z_j\big)\Big)\nonumber\\
&\quad\times R_2\Big(\vartheta(z_1+z_2)-u_2\big(1+\vartheta\sum z_j\big)\nonumber\\
&\qquad\qquad\qquad\qquad-u_3\big(u_1-u_2+\vartheta(-x_1-x_2+z_1+z_2+u_1\sum x_j-u_2\sum z_j)\big)\Big)\nonumber\\
&\quad\times R_2\Big(\vartheta(z_3+z_4)-u_2\big(1+\vartheta\sum z_j\big)\\
&\qquad\qquad\qquad\qquad-u_4\big(u_1-u_2+\vartheta(-x_3-x_4+z_3+z_4+u_1\sum x_j-u_2\sum z_j)\big)\Big)\nonumber\\
&\quad\times\iiiint_{\substack{t_1+t_3,t_2+t_4\leq 1\\t_1+t_2,t_3+t_4\leq 1}} P(1-t_1-t_3+x_1+z_1)P(1-t_2-t_4+x_2+z_2)\nonumber\\
&\quad\times P(1-t_1-t_2+x_3+z_3)P(1-t_3-t_4+x_4+z_4)dt_1dt_2dt_3dt_4du_1du_2du_3du_4\bigg|_{\underline{x}=\underline{z}=\underline{0}}\nonumber\\
&\quad+O_\varepsilon\big(T(\log T)^{3+\varepsilon}\big).\nonumber
 \end{align}

\section{Deduction of Theorem \ref{thm1} and Theorem \ref{thm2}}\label{proofthm12}

Theorem \ref{thm2} is a direct consequence of \eqref{keyinequality+}, \eqref{keyinequality-}, \eqref{asympS1} and \eqref{generalboundTkl}.

For Theorem \ref{thm1} we choose $P(x)=x^2$, and obtain from \eqref{asympS1} and \eqref{asympS2} that
\[
\mathcal{S}_{0,2}=-\Big(\frac{1}{4}+\frac{1}{9\vartheta}\Big)T(\log T)^2+O\big(T(\log T)\big)
\]
and
\[
\mathcal{T}_{0,2}=\Big(\frac{52\vartheta}{1215}+\frac{491}{5040}+\frac{563}{6300\vartheta}+\frac{659}{16200\vartheta^2}+\frac{8}{945\vartheta^3}+\frac{1}{1512\vartheta^4}\Big)T(\log T)^4+O_\varepsilon\big(T(\log T)^{3+\varepsilon}\big).
\]
The theorem now follows from \eqref{keyinequality-} and the choice $\vartheta=1/8-\varepsilon$.

\section{Deduction of Theorem \ref{thmHardy} and Corollary \ref{corHardy}}\label{proofthm3}

We recall from \eqref{Z'tformula} that
\begin{align*}
Z^{(k)}(t)&=i^ke^{i\theta(t)}(\log T)^kR_k\Big(\frac{1}{\log T}\frac{d}{d\alpha}\Big)\zeta(\tfrac{1}{2}+\alpha+it)\Big|_{\alpha=0}\\
&\qquad\qquad\qquad+O\bigg(\sum_{j=0}^{k}(\log T)^{j-1}|\zeta^{k-j}(\tfrac12+it)|\bigg).\nonumber
\end{align*}
Using the fact that $Z^{(k)}(t)=(-1)^kZ^{(k)}(-t)$ we hence obtain that
\begin{align*}
&\int_T^{2T}Z^{(k)}(t)Z^{(\ell)}(t)Z^{(m)}(t)Z^{(n)}(t)dt=(-1)^{m+n}i^{k+\ell+m+n}(\log T)^{k+\ell+m+n}\\
&\quad \times R_k\Big(\frac{1}{\log T}\frac{d}{d\alpha_1}\Big)R_\ell\Big(\frac{1}{\log T}\frac{d}{d\alpha_2}\Big)R_m\Big(\frac{1}{\log T}\frac{d}{d\beta_1}\Big)R_n\Big(\frac{1}{\log T}\frac{d}{d\beta_2}\Big)J(\alpha_1,\alpha_2,\beta_1,\beta_2)\bigg|_{\underline{\alpha}=\underline{\beta}=\underline{0}}\\
&\quad+O\big(T(\log T)^{k+\ell+m+n+3}\big).
\end{align*}
As argued in Section \ref{sectionTkl}, it follows from Proposition \ref{shiftedfourth} that
\begin{align*}
&\text{HARDY}(k,\ell,m,n)=(-1)^{m+n}i^{k+\ell+m+n}3\int_0^1\int_0^1\int_0^1\int_0^1(u_1-u_2)^2R_k\big((u_1-u_2)u_3-u_1\big)\\
&\qquad \times R_\ell\big((u_2-u_1)u_3-u_2\big)R_m\big((u_1-u_2)u_4-u_1\big)R_n\big((u_2-u_1)u_4-u_2\big)du_1du_2du_3du_4,
\end{align*}
which proves Theorem \ref{thmHardy}.

The above expression is equal to
\begin{align*}
&(-1)^{m+n}i^{k+\ell+m+n}3\int_0^1\int_0^1\int_0^1\int_0^1(u_1-u_2)^2\Big(\frac12+(u_1-u_2)u_3-u_1\Big)^k\\
&\ \times \Big(\frac12+(u_2-u_1)u_3-u_2\Big)^\ell \Big(\frac12+(u_1-u_2)u_4-u_1\Big)^m\Big(\frac12+(u_2-u_1)u_4-u_2\Big)^ndu_1du_2du_3du_4\\
&\ =i^{k+\ell+m+n}3\int_0^1\int_0^1\int_{-1/2}^{1/2}\int_{-1/2}^{1/2}(u_1-u_2)^2\big(u_1(u_3-1)-u_2u_3\big)^k\big(u_2(u_3-1)-u_1u_3\big)^\ell\\
&\ \times \big(u_2u_4-u_1(u_4-1)\big)^m\big(u_1u_4-u_2(u_4-1)\big)^n du_1du_2du_3du_4,
\end{align*}
by some change of variables. Expanding the integrand in powers of $u_1$ and $u_2$ we see that in the case $k+\ell+m+n$ is odd, the integrand is an odd function with respect to either $u_1$ or $u_2$. It follows that $\text{HARDY}(k,\ell,m,n)=0$, which proves Corollary \ref{corHardy}.\\


\noindent\textbf{Acknowledgements.}\ \ We would like to thank Micah Milinovich for some helpful discussions and for drawing our attention to Figures 1 and 2, and to the anonymous referee for drawing our attention to some related works.

\end{document}